\documentclass[a4paper,12pt,openany,oneside]{article}      

\usepackage{ucs}
\usepackage[utf8x]{inputenc}
\usepackage[english]{babel}

\usepackage[none]{hyphenat}      

\usepackage{amsmath}
\usepackage{amsthm}
\usepackage{amsfonts, amssymb}
\usepackage{yfonts} 

\usepackage{amscd} 
\usepackage[all,cmtip]{xy} 

\usepackage{graphicx}
\usepackage[final]{hyperref}

\theoremstyle{plain}

\newtheorem{theorem}{Theorem}[section]     

\newtheorem{proposition}[theorem]{Proposition}
\newtheorem{corollary}[theorem]{Corollary}
\newtheorem{lemma}[theorem]{Lemma}
\newtheorem*{affirmation}{Claim}
\newenvironment{proofitemenv}
  {\par\noindent\vspace{10pt}\rule{1ex}{1ex}\hspace{2ex}} 
  {\par}
\newtheorem{case}{Case}

\newtheorem{definition}[theorem]{Definition}

\theoremstyle{definition}

\newlength{\affrule}

\providecommand{\refeq}[1]{{(\ref{#1})}}

\providecommand{\naturals}[0]{\mathbb{N}}
\providecommand{\integers}[0]{\mathbb{Z}}

\providecommand{\reals}[0]{\mathbb{R}}

\providecommand{\complexes}[0]{\mathbb{C}}

\providecommand{\family}[1]{{\mathcal{#1}}}

\providecommand{\liealgebra}[1]{\mathfrak{#1}}

\providecommand{\suchthat}[0]{{\thickspace | \thickspace}}

\providecommand{\setsuchthat}[2]{{\left \{ {#1} \suchthat {#2} \right \}}}
\providecommand{\set}[1]{{\left \{ {#1} \right \}}}
\providecommand{\abs}[1]{{\lvert#1\rvert}}

\newcommand{\integral}[3][]{{\int_{#1} {#2}\, \mathrm{d}{#3}}}
\newcommand{\conditionalexpectation}[2]{{E \left({#1} | {#2}\right)}}

\DeclareMathOperator{\closureop}{cl}
\DeclareMathOperator{\supportop}{supp}

\newcommand{\function}[3]{{#1: #2 \to #3}}

\newcommand{\functionarray}[5]{{\begin{array}[t]{lrcl}
                                  #1:& #2 &\to     &#3 \\
                                     & #4 &\mapsto &#5
                                \end{array}}}

\newcommand{\closure}[1]{{\closureop \left({#1}\right)}}

\newcommand{\complementset}[1]{{#1}^{c}}

\newcommand{\support}[1]{{\supportop \left( {#1} \right)}}

\DeclareMathOperator{\Gl}{Gl}

\DeclareMathOperator{\Ad}{Ad}

\newcommand{\recurrent}{{\mathcal{R}}}

\DeclareMathOperator{\fix}{fix}

\newcommand{\Matrix}[2]{\left( \begin{array}{#1}#2\end{array} \right)}

\newcommand{\covercardinality}[2][]{{N_{#1}\left({#2}\right)}}

\newcommand{\mensurablesets}[1]{{\family{B}}}

\newcommand{\probability}[1]{{\mu}}
\newcommand{\measure}[1]{{\mu}}

\newcommand{\measureentropy}[2]{{h_{#1}\left({#2}\right)}}

\newcommand{\tcoverentropy}[2]{{h\left({#2}\,|\,{#1}\right)}}

\newcommand{\topologicalentropy}[1]{{h\left({#1}\right)}}

\newcommand{\ksentropy}[1]{{\sup_{\mu} \measureentropy{\mu}{T}}}

\newcommand{\bowenentropy}[2]{{h_{#1}\left({#2}\right)}}

\newcommand{\pinsker}[1][]{{\family{P}_{#1}}}

\begin{document}

\title{Entropy of Endomorphisms of Lie Groups}

\author{
Andr\'{e} Caldas\footnote{Departamento de Matem\'{a}tica - Universidade
de Bras\'{\i}lia-DF, Brasil. Supported by CNPq grant no.\ 140888/11-0}\, and Mauro Patr\~{a}o\footnote{Departamento de Matem\'{a}tica -
Universidade de Bras\'{\i}lia-DF, Brasil. Supported by CNPq grant no.\ 310790/09-3}
}

  \maketitle

    \begin{abstract}
    We show, when $G$ is a nilpotent or reductive Lie
    group, that the entropy of any surjective endomorphism
    coincides with the entropy of its restriction to the toral component of the center of $G$. In particular, if $G$ is a semi-simple Lie group, the entropy of any surjective endomorphism always vanishes.
    Since every compact group is reductive, the formula for the entropy of a endomorphism of a compact group reduces to the formula for the entropy of an endomorphism of a torus.
    We also characterize the recurrent set of conjugations of linear semi-simple Lie groups.
    \end{abstract}



\section{Introduction}

  In \cite{handel} and \cite{patrao:entropia}, it is introduced a topological notion of entropy for proper maps on locally compact separable metrizable spaces.
  It is shown there that this topological entropy coincides with the
  supremum of the Kolmogorov-Sinai's entropies and also with the minimum of the  Bowen's entropies.
  Using this variational principle, it is also shown in \cite{patrao:entropia} that the
  topological entropy of a linear isomorphism of a finite dimensional
  vector space always vanishes.
  This shows that the Bowen's formula (see \cite{bowen:entropia}) for
  the entropy of an endomorphism of a non-compact Lie Group gives just
  an upper bound for its topological entropy.
  At the end of \cite{patrao:entropia}, using again the variational
  principle, it is proved that the topological entropy of the endomorphism $\phi(z) = z^2$, where $z \in \complexes^*$, is equal to the topological entropy of its restriction to $S^1 \subset \complexes^*$.
  Using the same kind of reasoning presented there, one can show that
  the same result holds for any endomorphism $\phi(z) = z^n$, where $n \in \naturals$.

  These examples led us to the following conjecture.
  Since every connected abelian Lie group $G$ is isomorphic to
  the product of a torus (the toral component of $G$) by a finite
  dimensional vector space and since the component of $\complexes^*$ is just $S^1$, the topological entropy of a proper endomorphism of a connected abelian Lie group $G$ should be just the topological entropy of its restriction to the toral component of $G$.

  More generally, for a connected Lie group $G$, we can consider the toral component of the identity component of the center of $G$, denoted by $T(G)$ and called in the present paper just \emph{the toral component of $G$}.
  From now on we will use the term \emph{entropy} to refer to the
  topological entropy.
  In this paper, when $G$ is a nilpotent or reductive Lie
  group, we show that the entropy of any surjective endomorphism
  coincides with the entropy of its restriction to $T(G)$.
  Since every compact group is reductive, these results shed new light in
  the Bowen's formula even in the compact case.
  In fact, the formula for the entropy of a endomorphism of a compact
  group reduces to the formula for the entropy of an endomorphism of a
  torus (see \cite{sinai:entropia}).
  In particular, if $G$ is a compact semi-simple Lie group, the entropy of any
  surjective endomorphism always vanishes.
  One may wonder, for a general connected Lie group $G$, that the entropy of a surjective endomorphism coincides with the entropy of its restriction to $T(G)$.
  At the end of this article, we present arguments that suggest that
  the general case is slightly different.


  The paper is structured in sections corresponding to the classes
  of Lie groups for which we compute the topological entropy of
  their respective endomorphisms.
  But first, in a preliminary section, we collect some results used
  in the remaining sections.
  Since the concept of
  topological entropy given in \cite{patrao:entropia}
  requires properness, we always start the remaining sections
  considering surjective endomorphisms and
  showing that these endomorphisms are in fact proper maps.
  In Section
  \ref{sec:caso_abeliano},
  we treat the abelian case and shows the above conjecture about entropy.
  Section
  \ref{sec:caso_nilpotente}
  treats the nilpotent case.
  As in the abelian case, we conclude
  that the entropy of an endomorphism of $G$ is the entropy
  of its restriction to $T(G)$.
  In Section
  \ref{sec:caso_semi_simples},
  the semi-simple case is considered.
  Using the relation in a linear semi-simple Lie group between
  endomorphisms, conjugations and linear maps,
  the recurrent set of a given conjugation
  is characterized.
  Then, the concept of Li-Yorke pairs is used to demonstrate that
  the entropy of endomorphisms of semi-simple Lie groups
  (even not linear ones) always vanishes.
  In particular, since $T(G)$ is trivial in this case,
  the entropy of an endomorphism coincides with the entropy of
  its restriction to $T(G)$.
  In Section
  \ref{sec:caso_redutivel},
  we compute the entropy of endomorphisms of reductive groups,
  by using the previous results for the abelian and semi-simple cases.
  And finally, in Section
  \ref{sec:caso_geral},
  we end the paper with some remarks about the general case.


\section{Preliminaries}

  In this section, we collect some facts that are used in the next sections.
  In general, we just state the facts with references, but we do not present much details of the theory involved.
  From \cite{patrao:entropia}, Theorem 3.2, we have the following
  variational principle, where $\topologicalentropy{\phi}$
  is the \emph{topological entropy} defined in \cite{patrao:entropia}.
  That is, $\topologicalentropy{\phi}$ is the supremum of
  the entropies taken over all \emph{admissible covers}.

  \begin{proposition}
    \label{prop:principio_variacional:patrao}
    Let $X$ be a locally compact metrizable separable space and
    $\function{\phi}{X}{X}$ a proper map.
    Then
    \begin{equation*}
      \sup_{\mu}
      \measureentropy{\mu}{\phi}
      =
      \topologicalentropy{\phi}
      =
      \min_{d}
      \bowenentropy{d}{\phi}.
    \end{equation*}
  \end{proposition}

  As a consequence of Proposition
  \ref{prop:principio_variacional:patrao}
  we present the following formula for the entropy of products, which is well known in the compact case.

  \begin{proposition}
    \label{prop:entropia:formula_do_produto}
    Let $X$ and $Y$ be locally compact metrizable separable spaces and
    $\function{\phi}{X}{X}$,
    $\function{\psi}{Y}{Y}$ proper maps.
    Then,
    \begin{equation*}
      \topologicalentropy{\phi \times \psi}
      =
      \topologicalentropy{\phi}
      +
      \topologicalentropy{\psi}.
    \end{equation*}
  \end{proposition}

  \begin{proof}
    From the variational principle, we have that
    \begin{align*}
      \topologicalentropy{\phi}
      +
      \topologicalentropy{\psi}
      &=
      \sup_{\mu}
      \measureentropy{\mu}{\phi}
      +
      \sup_{\nu}
      \measureentropy{\nu}{\psi}
      \\
      &=
      \sup_{\mu,\nu}
      \measureentropy{\mu \times \nu}{\phi \times \psi}
      \leq
      \topologicalentropy{\phi \times \psi}.
    \end{align*}
    On the other hand, by the variational principle there exist metrics $d_1$ and $d_2$ such that
    \begin{equation*}
      \topologicalentropy{\phi}
      = \bowenentropy{d_1}{\phi}
      \qquad \mbox{and} \qquad
      \topologicalentropy{\psi}
      =
      \bowenentropy{d_2}{\psi}.
    \end{equation*}
    But Proposition 2.2.15 of \cite{ferraiol:entropia} states that for
    the maximum distance $d$,
    \begin{equation*}
      \bowenentropy{d}{\phi \times \psi}
      \leq
      \bowenentropy{d_1}{\phi}
      +
      \bowenentropy{d_2}{\psi}.
    \end{equation*}
    Now, the variational principle leads us to
    \begin{equation*}
      \topologicalentropy{\phi \times \psi}
      \leq
      \bowenentropy{d_1}{\phi}
      +
      \bowenentropy{d_2}{\psi}
      =
      \topologicalentropy{\phi}
      +
      \topologicalentropy{\psi}.
    \end{equation*}
  \end{proof}

  The following proposition, which is used in section
  \ref{sec:caso_redutivel},
  also generalizes a simple result from the compact case
  (see Proposition 2.1 of \cite{patrao:entropia}).

  \begin{proposition}
    \label{prop:entropia:conjugacao}

    Let $X$ and $Y$ be locally compact metrizable separable spaces, and
    consider the diagram
    \begin{equation*}
      \xymatrix
      {
        X \ar[d]_{\pi} \ar[r]^\phi  &X \ar[d]^{\pi}
        \\
        Y \ar[r]_{\psi}             &Y
      },
    \end{equation*}
    where $\phi$, $\psi$ and $\pi$ are proper surjective maps.
    Then,
    \begin{equation*}
      \topologicalentropy{\psi}
      \leq
      \topologicalentropy{\phi}.
    \end{equation*}
  \end{proposition}

  The following result is proved in Corollary 16 of
  \cite{bowen:entropia}.

  \begin{proposition}[Bowen's Formula]
    If $\phi$ is an endomorphism of a Lie group $G$
    and $d$ is a right invariant distance, then
    \begin{equation*}
      \bowenentropy{d}{\phi}
      =
      \sum_{\abs{\lambda} > 1}
      \log \abs{\lambda},
    \end{equation*}
    where $\lambda$ are the eigenvalues of $\phi'$, the differential of
    $\phi$ at the identity, counted with multiplicity.
  \end{proposition}

  The following result about compact principal bundles is proved in Theorem 19 of \cite{bowen:entropia}.

  \begin{proposition}
    \label{prop:entropia_soma_em_espacos_homogeneos}

    Let $\pi: X \to Y$ be a compact metrizable $G$-principal bundle. Assume that $\function{\phi}{X}{X}$, $\function{\psi}{Y}{Y}$, and
    $\function{\tau}{G}{G}$ are continuous maps so that
    $\psi \circ \pi = \pi \circ \phi$ and $\phi(xg) = \phi(x)\tau(g)$. Then
    \begin{equation*}
      \topologicalentropy{\phi}
      =
      \topologicalentropy{\psi}
      +
      \topologicalentropy{\tau}.
    \end{equation*}
  \end{proposition}

  A tool we used in order to show that certain systems have
  zero entropy is the concept of Li-Yorke pairs.
  In this paper, we introduce a topological definition of a Li-Yorke pair,
  which is more adjusted to the non-compact case
  then the usual definition wich uses $\limsup$ and $\liminf$
  of the distances between $\phi^n(x)$ and $\phi^n(y)$
  (see page 2 of \cite{on_li_yorke_pairs}).

  \begin{definition}[Li-Yorke pair]
    \label{def:li-yorke_pair}

    Given a continuous function $\function{\phi}{X}{X}$,
    $(a,b) \in X \times X$ is called a \emph{Li-Yorke pair} when there exist
    $c \in X$, $n_k \rightarrow \infty$ and
    $m_k \rightarrow \infty$ such that
    $(\phi^{n_k}(a), \phi^{n_k}(b)) \rightarrow (a,b)$ and
    $(\phi^{m_k}(a), \phi^{m_k}(b)) \rightarrow (c,c)$.
  \end{definition}

  We shall prove that positive entropy implies the existence
  of Li-Yorke pairs.
  A stronger version could be demonstrated along the lines of
  Theorem 2.3 in \cite{on_li_yorke_pairs}, but since we
  do not need that level of generality, we shall prove this simpler version.

  \begin{definition}[Pinsker $\sigma$-algebra and product measure]
    For a $\phi$-invariant measure $\mu$,
    the \emph{Pinsker $\sigma$-algebra} $\pinsker{}$,
    is greatest sub-$\sigma$-algebra such that
    $\measureentropy{\mu|_{\pinsker{}}}{\phi} = 0$.
    The \emph{product} of $\mu$ with itself over $\pinsker{}$,
    denoted by $\lambda_\mu$, is the measure on $X \times X$
    such that, for any measurable sets $A, B \subset X$,
    \begin{equation*}
      \lambda_\mu(A \times B)
      =
      \integral{\conditionalexpectation{1_A}{\pinsker{}}
                \conditionalexpectation{1_B}{\pinsker{}}}{\mu}.
    \end{equation*}
  \end{definition}

  Notice that if $\mu$ is $\phi$-ergodic, then $\lambda_\mu$ is
  $\phi \times \phi$-ergodic
  (see \cite{entropy_pairs}).
  Also, notice that if $\mu(A) = 0$, then $A \in \pinsker{}$.

  \begin{lemma}
    \label{le:diagonal_com_medida_nula}

    Denote by $\Delta$, the set of diagonal points in $X \times X$.
    That is, $\Delta = \setsuchthat{(x,x)}{x \in X}$.
    Then, for a $\phi$-ergodic measure $\mu$ such that
    $\measureentropy{\mu}{\phi} > 0$,
    \begin{equation*}
      \lambda_\mu(\Delta) = 0.
    \end{equation*}
  \end{lemma}

  \begin{proof}
    Since $\lambda_\mu$ is ergodic
    and $\Delta$ is $\phi \times \phi$-invariant,
    it is enough to prove that
    $\lambda_\mu(\Delta) \neq 1$.
    Suppose on the contrary, that
    $\lambda_\mu(\Delta) = 1$.
    Then, for any measurable $B \subset X$,
    \begin{equation*}
      \integral{\conditionalexpectation{1_B}{\pinsker{}}
                \conditionalexpectation{1_{\complementset{B}}}{\pinsker{}}}{\mu}
      =
      \lambda_\mu
      \left(B \times \complementset{B}\right)
      =
      0.
    \end{equation*}
    And since the conditional expectation of a non-negative function is
    non-negative, we have that
    \begin{equation}
      \label{eq:excludent_conditionals}
      \conditionalexpectation{1_B}{\pinsker{}}
      \conditionalexpectation{1_{\complementset{B}}}{\pinsker{}}
      =
      0
    \end{equation}
    almost everywhere.
    Let
    \begin{equation*}
      A
      =
      \setsuchthat{x \in X}{\conditionalexpectation{1_B}{\pinsker{}} \neq 0}.
    \end{equation*}
    Then, $A \in \pinsker{}$.
    Now,
    \begin{align*}
      \mu(B \setminus A)
      &=
      \integral[\complementset{A}]{1_B}{\mu}
      \\
      &=
      \integral[\complementset{A}]{\conditionalexpectation{1_B}{\pinsker{}}}{\mu}
      \\
      &=
      0.
    \end{align*}
    Likewise,
    \begin{align*}
      \mu(A \setminus B)
      &=
      \integral[A]{1_{\complementset{B}}}{\mu}
      \\
      &=
      \integral[A]{\conditionalexpectation{1_{\complementset{B}}}{\pinsker{}}}{\mu}
      \\
      &=
      0,
    \end{align*}
    where the last equality is due to the fact that
    when restricted to $A$,
    $\conditionalexpectation{1_B}{\pinsker{}} \neq 0$,
    and therefore, according to equation
    \refeq{eq:excludent_conditionals},
    $\conditionalexpectation{1_{\complementset{B}}}{\pinsker{}} = 0$
    almost everywhere.

    Using those relations, since $\pinsker{}$ is a sub-$\sigma$-algebra
    containing $A$, and containing also all measurable sets with null
    $\mu$ measure, we conclude that $B \in \pinsker{}$.
    But since $B$ was arbitrary, this would imply that
    $\measureentropy{\mu}{\phi} = 0$.
    The contradiction implies that
    $\lambda_\mu(\Delta) \neq 1$.
  \end{proof}

  \begin{lemma}
    \label{le:diagonal_no_suporte}

    Let $Z = \support{\mu}$ and $W = \support{\lambda_\mu}$.
    Denote by $\Delta_Z$ the set of diagonal points $(z,z)$,
    with $z \in Z$.
    That is, $\Delta_Z = \setsuchthat{(z,z)}{z \in Z}$.
    Then,
    \begin{equation*}
      \Delta_Z
      \subset W.
    \end{equation*}
  \end{lemma}

  \begin{proof}
    If $(z,z) \in Z \times Z$ is not in the support of $\lambda_\mu$,
    then, there exists an open neighborhood $U$ of $(z,z)$,
    such that $\lambda_\mu(U) = 0$.
    In this case, there exists an open neighborhood $A$ of $z$,
    such that $z \in A \times A \subset U$.
    Then,
    \begin{align*}
      \lambda_\mu(U) = 0
      &\Rightarrow
      \lambda_\mu(A \times A) = 0
      \\
      &\Leftrightarrow
      \integral{\conditionalexpectation{1_A}{\pinsker{}}^2}{\mu} = 0
      \\
      &\Leftrightarrow
      \integral{\conditionalexpectation{1_A}{\pinsker{}}}{\mu} = 0
      \\
      &\Leftrightarrow
      \mu(A) = 0.
    \end{align*}
    But this contradicts the fact that $z$ is in the support of $\mu$.
  \end{proof}

  \begin{proposition}
    \label{prop:positive_entropy_has_li-yorke_pair}

    Let $X$ be a locally compact metrizable separable space and
    $\function{\phi}{X}{X}$ a proper map.
    If $\topologicalentropy{\phi} > 0$, then there exists a Li-Yorke
    pair for this system.
  \end{proposition}

  \begin{proof}
    The fact that $\topologicalentropy{\phi} > 0$ implies,
    using Proposition
    \ref{prop:principio_variacional:patrao},
    that there exists a $\phi$-ergodic measure $\mu$,
    such that
    \begin{equation*}
      \measureentropy{\mu}{\phi} > 0.
    \end{equation*}

    Let $Z = \support{\mu}$ and $W = \support{\lambda_\mu}$.
    For each non-empty open set $U$ of $W$ in its relative topology,
    let
    \begin{equation*}
      W_U
      =
      \limsup_{n \rightarrow \infty}
      \phi^{-n}(U).
    \end{equation*}
    The ergodicity of $\lambda_\mu$ implies that
    $\lambda_\mu(W_U) = 1$.
    Therefore, if $\beta$ is an enumerable basis for the topology
    of $W$, then $\lambda_\mu\left( W_\beta \right) = 1$ for
    \begin{equation*}
      W_\beta
      =
      \bigcap_{U \in \beta}
      W_U.
    \end{equation*}
    Together with Lemma
    \ref{le:diagonal_com_medida_nula},
    this implies that
    \begin{equation*}
      \lambda_\mu\left(W_\beta \setminus \Delta\right)
      =
      1.
    \end{equation*}
    In particular, for every point $(a,b) \in W_\beta$,
    with $a \neq b$, there is a sequence
    $n_k$ such that $(T^{n_k}a, T^{n_k}b) \rightarrow (a,b)$.
    And also, by Lemma
    \ref{le:diagonal_no_suporte},
    if we pick $c \in Z$, then $(c,c) \in W$.
    And therefore, there is a sub-sequence
    $m_k$ such that $(T^{m_k}a, T^{m_k}b) \rightarrow (c,c)$.
    That is, any such pair $(a,b)$ is a Li-Yorke pair.
  \end{proof}

  The following lemma is a simple topological fact.

  \begin{lemma}
    \label{lemma:endomorfismo_com_nucleo_compacto_eh_proprio}

    Let $\function{\phi}{G}{G}$ be a surjective endomorphism of a
    Lie group $G$.  Then,
    \begin{equation*}
      \phi \text{ is proper}
      \Leftrightarrow
      \ker(\phi) \text{ is compact}.
    \end{equation*}
  \end{lemma}

  \begin{proof}
    Being $\phi$ a proper mapping, $\ker(\phi)$ is evidently compact,
    since it is the inverse image of $\set{1_G}$.
    On the other hand, a surjective endomorphism is always an open
    mapping.  If $\ker(\phi)$ is compact,
    then $\phi$ is a continuous surjection with compact fibers.
    It is a known fact that any continuous open surjection with
    compact fibers is a proper mapping.
  \end{proof}

  We end this preliminary section with the following characterization of the recurrent set of a linear isomorphism of a finite dimensional vector space proved in Proposition 4.2 in \cite{patrao:entropia}.

  \begin{proposition}
    \label{prop:recorente isomorfismo linear}

    Let $\function{T}{V}{V}$ be a linear isomorphism of a finite dimensional vector space and consider its multiplicative Jordan decomposition $T = T_E T_H T_U$. The recurrent set of $T$ is given by
  \begin{equation*}
    \recurrent(T)
    =
    \fix(T_H)
    \cap
    \fix(T_U).
  \end{equation*}
  \end{proposition}

\section{The Abelian Case}
  \label{sec:caso_abeliano}

  In this section, we determine the entropy of a surjective endomorphism $\phi$ of connected abelian Lie group $G$. In this case, the exponential map is a
  surjective group homomorphism with discrete kernel.
  Therefore, $G$ can be identified with $T(G) \times \reals^q$,
  where $T(G)$ is isomorphic to the $p$-dimensional torus
  $\reals^p / \integers^p$ and the group operation is addition.
  The endomorphisms of $G$ can be identified with linear maps of the
  form
  \begin{equation*}
    \phi
    =
    \Matrix{cc}
    {
      T & *
      \\
      0 & S
    },
  \end{equation*}
  where $\function{T}{\reals^p}{\reals^p}$ is a linear map that leaves
  $\integers^p$ invariant, and $\function{S}{\reals^q}{\reals^q}$ is a
  linear isomorphism.
  Notice that for $(x,y) \in T(G) \times \reals^q$, the action of
  $\phi$ have the form
  \begin{equation*}
    \phi^n(x,y) = (*,S^n y).
  \end{equation*}
  In particular, if $(x,y)$ is a recurrent point,
  then $y \in \recurrent(S)$ and thus
  \begin{equation*}
    \closure{\recurrent(\phi)} \subset T(G) \times \closure{\recurrent(S)}.
  \end{equation*}
  But since $S$ and $S_E$ coincide in $\closure{\recurrent(S)}$, we have that $\phi$ and
  \begin{equation*}
    \Matrix{cc}
    {
      T & *
      \\
      0 & S_E
    }
  \end{equation*}
  coincide in $T(G) \times \closure{\recurrent(S)}$.
  And since $S_E$ has only eigenvalues with modulus $1$, Bowen's
  formula and the variational principle gives that
  \begin{equation*}
    \topologicalentropy{\phi}
    \leq
    \bowenentropy{d}{\phi}
    =
    \bowenentropy{d}{\phi|_{T(G)}}
    =
    \topologicalentropy{\phi|_{T(G)}}
    \leq
    \topologicalentropy{\phi},
  \end{equation*}
  where $d$ is an invariant distance.
  That is, the topological entropy of $\phi$ is just the
  topological entropy of $\phi$ restricted to its toral component.


\section{The Nilpotent Case}
  \label{sec:caso_nilpotente}


\subsection{Properness}

  We start proving that every surjective endomorphism $\phi$ of a connected
  nilpotent Lie group $G$ is a proper map.
  Since the $\phi'$ is a surjective linear endomorphism of the Lie
  algebra, it is a linear isomorphism and thus a proper map.
  If the $G$ is simply-connected nilpotent, we have $\phi$ is conjugated to
  its differential at the identity $\phi'$ and the following diagram commutes
  \begin{equation*}
    \xymatrix
    {
      \liealgebra{g} \ar[d]_{\exp}  \ar[r]^{\phi'}                     & \liealgebra{g} \ar[d]^{\exp} \\
      G  \ar[r]_{\phi}  & G
    }
  \end{equation*}
  Thus $\phi$ is an automorphism and a proper map. We can reduce the general connected nilpotent case to the simply-connected one. But first we need the following lemmas.

  \begin{proposition}\label{propGTGsimplesmenteconexo}
    Let $G$ be a nilpotent Lie group and $T(G)$ its toral component.
    Then $G/T(G)$ is simply connected.
  \end{proposition}

  \begin{proof}
    Let $\pi: \widetilde{G} \to G$ be the universal covering of $G$. We have that $G$ is isomorphic to $\widetilde{G} / \ker(\pi)$. Besides this, we have that $Z(\widetilde{G})$ isomorphic to a finite dimensional vector space and that
    \begin{equation*}
    Z(G)_0 = Z(G) = \pi\left(Z(\widetilde{G})\right).
    \end{equation*}
    Therefore $Z(G)_0$ is isomorphic to $Z(\widetilde{G})/ \ker(\pi)$ and thus $T(G)$ is isomorphic to $\widetilde{T} / \ker(\pi)$, where $\widetilde{T} = \pi^{-1}\left(T(G))\right)$ is isomorphic to a vector subspace. On the other hand, we have that
    \begin{equation*}
    \frac{G}{T(G)} \simeq \frac{\widetilde{G}/\ker(\pi)}{\widetilde{T} / \ker(\pi)} \simeq \frac{\widetilde{G}}{\widetilde{T}},
    \end{equation*}
    which shows that $G/T(G)$ is simply connected, since $\widetilde{G}/\widetilde{T}$ is homeomorphic to the quotient of a vector space by a vector subspace.
  \end{proof}

  \begin{proposition}
    \label{prop:endomorfismo_sobrejetivo_de_nilpotente_eh_proprio}

    Let $\function{\phi}{G}{G}$ be a surjective endomorphism of a
    nilpotent connected Lie group $G$.
    Then $\phi$ is a proper map.
  \end{proposition}

  \begin{proof}
  By Proposition \ref{propGTGsimplesmenteconexo}, we have that $\widetilde{G} = G/T(G)$ is simply connected. Besides this, for any surjective
  endomorphism $\function{\phi}{G}{G}$, there is a surjective endomorphism
  $\function{\widetilde{\phi}}{\widetilde{G}}{\widetilde{G}}$ such that the following diagram commutes
  \begin{equation*}
    \xymatrix{
      G      \ar[d]_{\pi} \ar[r]^\phi &G \ar[d]^{\pi}
      \\
      \widetilde{G} \ar[r]_{\widetilde{\phi}}      & \widetilde{G}
    }
  \end{equation*}
  where $\pi$ is the canonical projection. In fact, since $\phi \left( T(G) \right)$ is compact and
  abelian, it is necessarily contained in $T(G)$.
  Thus we can define
  \begin{equation*}
  \widetilde{\phi}(\pi(x)) = \pi(\phi(x)) = \phi(x) T(G).
  \end{equation*}
  By the above discussion, since $\widetilde{G}$ is simply connected, we have that $\widetilde{\phi}$ is an automorphism of $\widetilde{G}$. Now, we claim that $\ker(\phi)$ is a closed subset of the compact set $T(G)$. In fact, we have that
    \begin{align*}
      \phi^{-1}
      \left(
        1_G
      \right)
      &\subset
      \phi^{-1}
      \left(
        T(G)
      \right)
      \\
      &=
      \phi^{-1}
      \left(
        \pi^{-1}(1_{\widetilde{G}})
      \right)
      \\
      &=
      \pi^{-1}
      \left(
        \widetilde{\phi}^{-1}(1_{\widetilde{G}})
      \right)
      \\
      &=
      \pi^{-1}
      \left(
        1_{\widetilde{G}}
      \right)
      \\
      &=
      T(G).
    \end{align*}
    Thus we have that $\ker(\phi)$ is compact and, by Lemma \ref{lemma:endomorfismo_com_nucleo_compacto_eh_proprio}, we have that $\phi$ is a proper map.
  \end{proof}


\subsection{Topological Entropy}

  In this subsection, we show that, as in the abelian case, the entropy
  of a surjective endomorphism of a nilpotent Lie group $G$ coincides with
  the entropy of its restriction to the toral component of $G$.

  \begin{theorem}\label{teoentropianilpotente}
    Let $\function{\phi}{G}{G}$ be a surjective endomorphism of a
    connected nilpotent Lie group $G$.
    Then
    \begin{equation*}
      \topologicalentropy{\phi}
      =
      \topologicalentropy{\phi|_{T(G)}}.
    \end{equation*}
  \end{theorem}

  \begin{proof}
    Let $\widetilde{G} = G/T(G)$, and denote by $\function{\pi}{G}{G/T(G)}$
    the canonical projection.
    By Proposition \ref{propGTGsimplesmenteconexo}, we have that $\widetilde{G}$ is a simply-connected nilpotent Lie group and, as in the proof of Proposition \ref{prop:endomorfismo_sobrejetivo_de_nilpotente_eh_proprio}, we can consider the induced
    endomorphism $\function{\widetilde{\phi}}{\widetilde{G}}{\widetilde{G}}$.  We have that $\function{\widetilde{\phi}}{\widetilde{G}}{\widetilde{G}}$ is    conjugated to its differential at the unity $\widetilde{\phi}'$ through
    the exponential map.

    On the other hand, we have that $\widetilde{\phi}'$ is a linear map and, by Proposition 4.2 in \cite{patrao:entropia},
    that its recurrent set
    $\recurrent(\widetilde{\phi}')$ is closed.
    We also know that there is a norm in $\widetilde{\liealgebra{g}}$ such that
    $\widetilde{\phi}'|_{\recurrent(\widetilde{\phi}')}$ is an isometry.
    In particular, for any closed ball
    $B \subset \widetilde{\liealgebra{g}}$ centered at $0$,
    $B \cap \recurrent(\widetilde{\phi}')$ is compact and $\widetilde{\phi}'$-invariant.
    From the conjugation given by the exponential map,
    there is a distance in $\recurrent(\widetilde{\phi})$
    such that any closed ball $B \subset \recurrent(\widetilde{\phi})$
    centered at the unit is compact $\widetilde{\phi}$-invariant.

    Let $R = \pi^{-1}(\recurrent(\widetilde{\phi}))$.
    Then, since $R$ is closed, it follows that
    $\closure{\recurrent(\phi)} \subset R$.
    For any $\varepsilon > 0$,
    there exists an admissible covering
    $\family{A} = \set{A_0, \dotsc, A_k}$ of
    $R$, such that
    \begin{equation*}
      \topologicalentropy{\phi}
      -
      \varepsilon
      =
      \topologicalentropy{\phi|_R}
      -
      \varepsilon
      \leq
      \tcoverentropy{\family{A}}{\phi_R}.
    \end{equation*}
    This admissible cover can be chosen in a way that $A_0$ has
    compact complement, and $A_1, \dotsc, A_k$ have compact closure,
    since, in a locally compact space, any admissible covering can be
    refined in this way.
    Let then $B \subset \recurrent(\widetilde{\phi})$
    be a compact $\widetilde{\phi}$-invariant ball such that
    $A_1, \dotsc, A_k$ all fall in $B$.
    Denoting $K = \pi^{-1}(B)$, it follows that
    $K$ is compact (since $\pi$ is proper),
    and $\covercardinality[R \setminus K]{\family{A}^n} = 1$
    (for $K$ is $\phi$-invariant and $K \subset A_0$).
    So,
    \begin{align*}
      \tcoverentropy{\family{A}}{\phi|_R}
      &\leq
      \tcoverentropy{\family{A} \vee \set{K, \complementset{K}}}{\phi|_R}
      \\
      &=
      \lim \frac{1}{n}
      \log
      \left(
        \covercardinality[K]{\family{A}^n}
        +
        \covercardinality[R \setminus K]{\family{A}^n}
      \right)
      \\
      &=
      \lim \frac{1}{n}
      \log
      \left(
        \covercardinality[K]{\family{A}^n}
        +
        1
      \right)
      \\
      &=
      \lim \frac{1}{n}
      \log
      \covercardinality[K]{\family{A}^n}
      \\
      &=
      \tcoverentropy{\family{A} \cap K}{\phi|_K}.
    \end{align*}
    Therefore,
    \begin{equation*}
      \topologicalentropy{\phi}
      -
      \varepsilon
      \leq
      \tcoverentropy{\family{A}}{\phi|_R}
      \leq
      \tcoverentropy{\family{A} \cap K}{\phi|_K}
      \leq
      \topologicalentropy{\phi|_K}.
    \end{equation*}
    On the other hand, applying Proposition
    \ref{prop:entropia_soma_em_espacos_homogeneos} for the compact $T(G)$-principal bundle $\pi|_K : K \to B$,
    since $\widetilde{\phi}$ is an isometry when restricted to
    $B$, we conclude that
    \begin{equation*}
      \topologicalentropy{\phi|_K}
      =
      \topologicalentropy{\widetilde{\phi}|_{B}}
      +
      \topologicalentropy{\phi|_{T(G)}}
      =
      \topologicalentropy{\phi|_{T(G)}}.
    \end{equation*}
    This way,
    \begin{equation*}
      \topologicalentropy{\phi}
      -
      \varepsilon
      \leq
      \topologicalentropy{\phi|_{T(G)}}
      \leq
      \topologicalentropy{\phi}.
    \end{equation*}
    And since $\varepsilon$ was arbitrary, it follows that
    \begin{equation*}
      \topologicalentropy{\phi}
      =
      \topologicalentropy{\phi|_{T(G)}}.
    \end{equation*}
  \end{proof}


\section{The Semi-simple Case}
  \label{sec:caso_semi_simples}


\subsection{Properness}

  We start proving that every surjective endomorphism $\phi$ of a connected
  semi-simple Lie group $G$ is a proper map. In fact, we show that, in this case, every surjective endomorphism is an automorphism.

  \begin{proposition}
    \label{prop:endomorfismo_sobrejetivo_de_semisimples_eh_automorfismo}

    Let $\function{\phi}{G}{G}$ be a surjective endomorphism of a
    semi-simple connected Lie group $G$.  Then there is a
    $k \in \naturals$ such that $\phi^k = C_g$ for some
    $g \in G$, where $C_g$ is the conjugation by $g$.
    In particular, $\phi$ is an automorphism.
  \end{proposition}

  \begin{proof}
    Notice that
    \begin{equation*}
      \phi^k ( \exp X )
      =
      \exp
      \left[
        (\phi')^k X
      \right].
    \end{equation*}
    But since $\liealgebra{g}$ is semi-simple, we know that there is a
    $k \in \naturals$ such that $(\phi')^k$ is an internal
    endomorphism of $\liealgebra{g}$
    (see Theorem 5.4, page 423 of \cite{helgason}).
    That is, there exists $g \in G$ such that $(\phi')^k = \Ad(g)$ and hence
    \begin{equation*}
      \phi^k ( \exp X ) = \exp((\phi')^k X) = \exp(\Ad(g) X) = C_g(\exp X).
    \end{equation*}
    Since $G$ is generated by elements
    of the form $\exp X$, it follows that $\phi^k = C_g$.

    We have that $\phi$ is an automorphism, since $C_g$ is an automorphism.
  \end{proof}

\subsection{Topological Entropy}

  In this subsection, we use the previous result in order to show that surjective   endomorphisms of connected semi-simple Lie groups have zero entropy.
  Proposition
  \ref{prop:endomorfismo_sobrejetivo_de_semisimples_eh_automorfismo},
  shows that some iteration of a surjective endomorphism $\phi$ of a semi-simple Lie group $G$ is in fact a conjugation by some element of $G$.
  Thus we first consider the dynamics of conjugations. For some $g \in G$, we denote by $G_g$ the centralizer of $g$ in $G$, which is the set of fixed points of the conjugation $C_g$.

  \begin{lemma}
    \label{lemma:recorrencia_em_grupos_lineares_semisimples}

    Let $G$ be a connected linear semi-simple Lie group,
    and $\function{C_g}{G}{G}$ the conjugation by $g \in G$.
    Then,
    \begin{equation*}
      \recurrent( C_g )
      =
      G_h
      \cap
      G_u,
    \end{equation*}
    where $g = ehu$ is the multiplicative Jordan decomposition of $g$.  In particular, $C_g$ restricted to its recurrent set is an isometry for some distance.
  \end{lemma}

  \begin{proof}
    Notice that $C_g = \Ad(g)|_G$, where $\Ad$ is the adjoint
    representation of $\Gl(d)$, with $G \leq \Gl(d)$.
    Also, Lemma 3.6 in \cite{patrao_santos_seco} shows that
    \begin{equation*}
      \Ad(g)
      =
      \Ad(e)
      \Ad(h)
      \Ad(u)
    \end{equation*}
    is the Jordan decomposition for $\Ad(g)$.

    Since the recurrent set behaves well with respect to restrictions, we know
    from Proposition \ref{prop:recorente isomorfismo linear} that
    \begin{align*}
      \recurrent(C_g)
      &=
      \recurrent(\Ad(g))
      \cap
      G
      \\
      &=
      \fix(\Ad(h))
      \cap
      \fix(\Ad(u))
      \cap
      G
      \\
      &=
      \fix(C_h)
      \cap
      \fix(C_u)
      \\
      &=
      G_h
      \cap
      G_u.
    \end{align*}
    The last claim follows, since $C_g$ coincides in $\recurrent( C_g )$ with the elliptic linear isomorphism $\Ad(e)$.
  \end{proof}

  \begin{theorem}
    \label{prop:semisimples_tem_entropia_nula}

    Let $\function{\phi}{G}{G}$ be a surjective endomorphism of a
    connected semi-simple Lie group $G$.
    Then
    \begin{equation*}
        \topologicalentropy{\phi} = \topologicalentropy{\phi|_{T(G)}} = 0.
    \end{equation*}
  \end{theorem}

  \begin{proof}
    Since $\topologicalentropy{\phi^k} = k\topologicalentropy{\phi}$, we have that $\topologicalentropy{\phi} = 0$ if and only if $\topologicalentropy{\phi^k} = 0$.  Since $G$ is connected and semi-simple, there is $k > 0$ such that $\phi^k = C_g$, for some $g \in G$.
    Therefore, it is enough to prove that
    $\topologicalentropy{C_g} = 0$.

    From proposition
    \ref{prop:positive_entropy_has_li-yorke_pair}, we know that, if $\topologicalentropy{C_g} > 0$, there exists a
    Li-Yorke pair for $C_g$, that is, two distinct elements
    $a, b \in G$, such that
    $(C_g^{n_k}(a), C_g^{n_k}(b)) \rightarrow (a,b)$ and
    $(C_g^{m_k}(a), C_g^{m_k}(b)) \rightarrow (c,c)$, for some
    $c \in G$.
    Consider $C_{\Ad(g)}$, and notice that
    $\Ad \circ C_g = C_{\Ad(g)} \circ \Ad$.
    Now, since $a,b \in \recurrent(C_g)$, we also have that
    $\Ad(a), \Ad(b) \in \recurrent(C_{\Ad(g)})$.
    But $C_{\Ad(g)}|_{\recurrent(C_{\Ad(g)})} = C_{\Ad(e)}|_{\recurrent(C_{\Ad(g)})}$
    is an isometry for some distance in $\Ad(G)$.
    This way, the fact that
    $\left(C_{\Ad(g)}^{m_k}(\Ad(a)), C_{\Ad(g)}^{m_k}(\Ad(b))\right)$ converges to
    $\left(\Ad(c), \Ad(c)\right)$ implies that $\Ad(a) = \Ad(b)$.

    So, we know that $a = w u$ and $b = w v$ for some $w \in G$
    and $u,v \in Z(G)$.
    We also have that
    \begin{align*}
      C_g^{m_k}(w) u
      =
      C_g^{m_k}(a)
      &\rightarrow
      c
      \\
      C_g^{m_k}(w) v
      =
      C_g^{m_k}(b)
      &\rightarrow
      c.
    \end{align*}
    But this means that $u = v$.
    And then $a = b$, contradicting the fact that they are a Li-Yorke
    pair.
  \end{proof}


\section{The Reductive Case}
  \label{sec:caso_redutivel}


\subsection{Properness}

  Let $G$ be a connected reductive Lie group.
  It will be useful to consider the surjective group homomorphism
  \begin{equation*}
    \functionarray{\pi}{Z(G)_0 \times G'}{G}
                       {(z\,\,, \,\,g)}{zg},
  \end{equation*}
  where $Z(G)_0$ is the identity component of its center, and
  $G' = [G,G]$ is the derived group which is connected and semi-simple.
  Also, $G$ and $Z(G)_0 \times G'$ have the same Lie algebra
  $\liealgebra{z} \times \liealgebra{g}'$, where
  $\liealgebra{g}$ is the Lie algebra of $G$,
  $\liealgebra{z}$ is the Lie algebra of $Z(G)_0$ and
  $\liealgebra{g}' = [\liealgebra{g},\liealgebra{g}]$
  is the Lie algebra of $G'$.

  \begin{lemma}
    \label{lemma:caso_redutivel_pode_ser_decomposto}

    Let $\function{\phi}{G}{G}$ be a surjective endomorphism of a connected
    reductive Lie group $G$.
    Then, $\phi$ induces the surjective homomorphism
    \begin{equation*}
      \widetilde{\phi} = \phi|_{Z(G)_0} \times \phi|_{G'}
    \end{equation*}
    in $Z(G)_0 \times G'$, such that the following diagram commutes
    \begin{equation*}
      \xymatrix
      {
        Z(G)_0 \times G'  \ar[d]_{\pi} \ar[r]^{\widetilde{\phi}}  &Z(G)_0 \times G' \ar[d]^{\pi}
        \\
        G                 \ar[r]_{\phi}                       &G
      },
    \end{equation*}
  \end{lemma}

  \begin{proof}
    It is evident that
    \begin{equation*}
      \phi(Z(G)_0)
      \subset
      Z(G)_0,
    \end{equation*}
    since $\phi(Z(G)_0)$ is a connected subgroup of $Z(G)$ containing the identity, and therefore is a subset of
    $Z(G)_0$.
    It happens that
    \begin{equation*}
      \phi(G')
      \subset
      G'
    \end{equation*}
    also holds.
    But this implies that
    $\phi'(\liealgebra{z}) = \liealgebra{z}$
    and
    $\phi'(\liealgebra{g}') = \liealgebra{g}'$.
    And because both $Z(G)_0$ and $G'$ are connected, we have
    that $\function{\phi|_{Z(G)_0}}{Z(G)_0}{Z(G)_0}$
    and $\function{\phi|_{G'}}{G'}{G'}$ are surjective.

    The commutativity of the above diagram is an immediate consequence of
    the fact that $\phi$ is an homomorphism.
  \end{proof}

  \begin{proposition}
    \label{lemma:caso_redutivel:aplicacao_eh_propria}

    Let $\function{\phi}{G}{G}$ be a surjective endomorphism of a connected
    reductive Lie group $G$. If $\pi$ is a proper map, then $\phi$ is a proper map.
  \end{proposition}

  \begin{proof}
    First observe that the endomorphism $\widetilde{\phi}$ presented in Lemma
    \ref{lemma:caso_redutivel_pode_ser_decomposto}
    is proper, since it is the product of two proper maps. In fact, we have that $\phi|_{Z(G)_0}$ and $\phi|_{G'}$ are proper endomorphisms, respectively, by Propositions
    \ref{prop:endomorfismo_sobrejetivo_de_nilpotente_eh_proprio}
    and
    \ref{prop:endomorfismo_sobrejetivo_de_semisimples_eh_automorfismo}.
    Considering the diagram in Lemma
    \ref{lemma:caso_redutivel_pode_ser_decomposto}, we have that, if $K \subset G$ is compact, then
    $\phi^{-1}(K) = \pi \circ \widetilde{\phi}^{-1} \circ \pi^{-1}(K)$
    is also compact, since $\widetilde{\phi}$ and $\pi$ are proper maps.
  \end{proof}


\subsection{Topological Entropy}
  We start computing the topological entropy of the endomorphism $\widetilde{\phi}$.

  \begin{proposition}
    \label{prop:redutivel:entropia_do_produto}

    Let $\function{\phi}{G}{G}$ be a surjective endomorphism of a
    reductive connected Lie group $G$ and $\widetilde{\phi}$ be the associated endomorphism. Then
    \begin{equation*}
      \topologicalentropy{\widetilde{\phi}}
      =
      \topologicalentropy{\phi|_{T(G)}}.
    \end{equation*}
  \end{proposition}

  \begin{proof}
    Using Proposition \ref{prop:entropia:formula_do_produto}, we have that
    \begin{equation*}
      \topologicalentropy{\widetilde{\phi}}
      =
      \topologicalentropy{\phi|_{Z(G)_0}}
      +
      \topologicalentropy{\phi|_{G'}}.
    \end{equation*}
    The result follows, since, from the abelian and semi-simple cases,
    we know that
    $\topologicalentropy{\phi|_{Z(G)_0}} = \topologicalentropy{\phi|_{T(G)}}$ and that
    $\topologicalentropy{\phi|_{G'}} = 0$.
  \end{proof}

  \begin{corollary}
    For any surjective endomorphism $\function{\phi}{G}{G}$ of a simply-connected reductive Lie group $G$,
    \begin{equation*}
      \topologicalentropy{\phi}
      =
      \topologicalentropy{\phi|_{T(G)}}.
    \end{equation*}
  \end{corollary}

  \begin{proof}
    Since $G$ is a universal covering and since the Lie algebras of $G$ and $Z(G)_0 \times G'$ coincide, the homomorphism $\pi$ is a conjugation between $\widetilde{\phi}$ and $\phi$.
  \end{proof}

    Now we consider the case where $G$ is not homeomorphic to $Z(G)_0 \times G'$.

  \begin{proposition}
    \label{prop:redutivel:projecao_do_produto_eh_propria}

    Let $\function{\phi}{G}{G}$ be a surjective endomorphism of a
    reductive connected Lie group $G$.
    If the projection $\function{\pi}{Z(G)_0 \times G'}{G}$ is proper, then
    \begin{equation*}
      \topologicalentropy{\phi}
      =
      \topologicalentropy{\phi|_{T(G)}}.
    \end{equation*}
  \end{proposition}

  \begin{proof}
    Consider the endomorphism $\tilde{\phi}$ from Lemma
    \ref{lemma:caso_redutivel_pode_ser_decomposto}.
    Now, use Proposition
    \ref{prop:entropia:conjugacao}
    and Proposition
    \ref{prop:redutivel:entropia_do_produto}
    to conclude that
    \begin{equation*}
      \topologicalentropy{\phi}
      \leq
      \topologicalentropy{\tilde{\phi}}
      =
      \topologicalentropy{\phi|_{T(G)}}
      \leq
      \topologicalentropy{\phi}.
    \end{equation*}
  \end{proof}

    As an immediate consequence, we solve the linear reductive case.

  \begin{corollary}
    Let $\function{\phi}{G}{G}$ be a surjective endomorphism of a connected
    linear reductive Lie group $G$.
    Then,
    \begin{equation*}
      \topologicalentropy{\phi}
      =
      \topologicalentropy{\phi|_{T(G)}}.
    \end{equation*}
  \end{corollary}

  \begin{proof}
    Using Proposition
    \ref{prop:redutivel:projecao_do_produto_eh_propria},
    we just have to show that $\pi$ is proper.
    But $\ker(\pi) = \setsuchthat{(x,x^{-1})}{x \in Z(G)_0 \cap G'}$.
    And, for a linear reductive Lie group $G$,
    $Z(G)_0 \cap G'$ is contained in the center of $G'$.
    But the center of a linear semi-simple Lie group is always finite.
    Now, we just use Lemma
    \ref{lemma:endomorfismo_com_nucleo_compacto_eh_proprio}.
  \end{proof}

    Another immediate consequence is that the formula for the entropy of a endomorphism of a compact group reduces to the formula for the entropy of an endomorphism of a torus.

  \begin{theorem}
    Let $\function{\phi}{G}{G}$ be a surjective endomorphism of a compact connected Lie group $G$.
    Then,
    \begin{equation*}
      \topologicalentropy{\phi}
      =
      \topologicalentropy{\phi|_{T(G)}}.
    \end{equation*}
  \end{theorem}

  \begin{proof}
    Since every compact Lie group is a reductive Lie group (see Proposition 6.6, page 132 of \cite{helgason}), we just have to show that $\pi$ is proper.
    But $Z(G)_0$ and $G'$ are compact subgroups of the compact group
    $G$ (see Theorem 6.9, page 133 of \cite{helgason}).
    Then $Z(G)_0 \times G'$ is compact and $\pi$ is proper.
  \end{proof}


\section{Remarks on the General Case}
  \label{sec:caso_geral}

    In this section, we present a conjecture about the entropy of surjective proper endomorphisms $\phi$ of a connected Lie group $G$. Based on the previous particular cases, one may conjecture that the entropy $\phi$ coincides with the entropy of its restriction to the toral component of $G$. But, if we could prove Proposition \ref{prop:entropia_soma_em_espacos_homogeneos} for locally compact principal bundles, we would first conclude that
    \begin{equation*}
      \topologicalentropy{\phi}
      =
      \topologicalentropy{\phi|_{R_0}},
    \end{equation*}
    where $R = R(G)_0$ is identity component of the solvable radical $R(G)$ of $G$. In fact, considering the diagram
    \begin{equation*}
      \xymatrix
      {
        G  \ar[d]_{\pi} \ar[r]^{\phi}  & G \ar[d]^{\pi}
        \\
        G/R                 \ar[r]_{\widetilde{\phi}}                       & G/R
      },
    \end{equation*}
    which is well defined, since $\phi(R) = R$, we would have that
     \begin{equation*}
      \topologicalentropy{\phi}
      =
      \topologicalentropy{\phi|_{R}} + \topologicalentropy{\widetilde{\phi}}.
    \end{equation*}
    Thus we get the above first formula, since $G/R$ is a connected semi-simple Lie group. Now, considering $R'$, the derived subgroup of $R$, and the diagram
    \begin{equation*}
      \xymatrix
      {
        R  \ar[d]_{\pi} \ar[r]^{\phi|_R}  & R \ar[d]^{\pi}
        \\
        R/R'                 \ar[r]_{\widetilde{\phi|_R}}                       & R/R'
      },
    \end{equation*}
    which is well defined, since $\phi(R') = R'$, we would have that
     \begin{equation*}
      \topologicalentropy{\phi|_{R}}
      =
      \topologicalentropy{\phi|_{R'}}
      +
      \topologicalentropy{\widetilde{\phi|_R}}.
    \end{equation*}
    Since $R'$ is connected nilpotent and $R/R'$ is connected abelian, putting all together, we would have that
     \begin{equation*}
      \topologicalentropy{\phi}
      =
      \topologicalentropy{\phi|_{T(R')}}
      +
      \topologicalentropy{\widetilde{\phi|_R}|_{T(R/R')}}.
    \end{equation*}
    We would conclude that the formula of the topological entropy of a surjective proper endomorphism of a connected Lie group would reduce to the formula of the topological entropy of a surjective endomorphism of a torus.

\end{document}